\documentclass[12pt]{amsart}
\usepackage{latexsym}
\usepackage{amsmath,amsfonts,amssymb,amsthm,mathtools,amscd,bm,mathabx}
\usepackage{graphics}
\usepackage{csquotes}
\usepackage{blindtext}
\usepackage{xcolor}
 2

\usepackage{hyperref}
\newtheorem{theorem}{Theorem}[section]
\newtheorem{prop}{Proposition}[section]
\newtheorem{lemma}[theorem]{Lemma}
\newtheorem{cor}[theorem]{Corollary}
\newtheorem{definition}[theorem]{Definition}
\newtheorem{example}[theorem]{Example}

\newtheorem{fact}[theorem]{Fact}
\newtheorem{Note}[theorem]{Note}
\newtheorem{remark}[theorem]{Remark}
\numberwithin{equation}{section}

\title[Timelike minimal surface]{Timelike minimal surface in $\mathbb{E}^3_1$ with arbitrary ends}

\author{Priyank Vasu}
\author{Rahul Kumar Singh}
\author{Subham Paul}

\address{Department of Mathematics, Indian Institute of Technology Patna, Bihta, Patna-801106, Bihar, India}

\thanks{Priyank Vasu: \href{mailto:priyank_2121ma16@iitp.ac.in}{priyank\_2121ma16@iitp.ac.in}}
\thanks{Rahul Kumar Singh: \href{mailto:rahulks@iitp.ac.in}{rahulks@iitp.ac.in}}
\thanks{Subham Paul: \href{mailto:subham_2021ma25@iitp.ac.in}{subham\_2021ma25@iitp.ac.in}}

\begin{document}
\begin{abstract}
In this paper, we show the existence of a timelike minimal surface with an arbitrary number of weak complete ends. Then, we discuss the asymptotic behaviour of the simple ends and the topology of the singularity set of the constructed timelike minimal surface.
\end{abstract}
\maketitle
	\section{Introduction}
	
	A timelike surface in Lorentz-Minkowski 3-space $\mathbb{E}^3_1:=(\mathbb R^3, dx^2+dy^2-dz^2)$ whose mean curvature function vanishes identically is called a timelike minimal surface, and a spacelike surface with vanishing mean curvature function is called a spacelike minimal surface or more commonly a maximal surface. The theory of maximal surfaces is well-studied, and much progress has already been seen in the direction of maximal surfaces when compared to the theory of timelike minimal surfaces (see \cite{umehara2006maximal,kumar2023genus}, etc)

One of the important tools to study the geometry of zero mean curvature surfaces in $\mathbb{E}^3_1$ is their Weierstrass-Enneper representation, which O. Kobayashi gave in \cite{Kobayashi} for maximal surfaces and by J.J. Konderack in \cite{Konderak} for timelike minimal surfaces using split-complex numbers.

In this paper, motivated by the work of P. Kumar and S.R.R. Mohanty in \cite{kumar2023genus}, where they have shown the existence of genus-zero complete maximal surfaces with an arbitrary number of ends, we have considered the problem of the existence of a timelike minimal surface with an arbitrary number of \textit{weak complete ends}.  Since a timelike surface by definition is Lorentzian, and the fact that completeness is not a natural property of a Lorentzian metric as it is for a Riemannian metric (this is mainly due to the absence of results like the Hopf-Rinow theorem), see page 32 in \cite{Weinstein}. 
Because of this reason, in our problem, we consider the notion of \textit{weak completeness} instead of \textit{completeness}.

 In order to show the existence of such a timelike minimal surface requires addressing the period condition as given in Theorem \ref{thm timelike}, which is significantly more challenging than addressing the corresponding period condition for maximal and minimal surfaces. This complexity arises due to the lack of analytical tools available in the split-complex analysis compared to complex analysis (like the unavailability of the Cauchy integral type formula). A timelike minimal surface with ends causes the domain to become disconnected (see Definition \ref{End definition}), making it hard to define the period condition, and hence, it is difficult to understand how a timelike minimal surface behaves across the disconnected parts. To solve this problem, we have exploited the relationships between complex, split-complex, and bicomplex numbers. We believe that this result (see Theorem \ref{main result}) will make an important contribution to the global theory of timelike minimal surfaces in $\mathbb{E}^3_1$.

 We have also shown that the sum of two maximal surfaces (with certain assumptions on their corresponding Weierstrass data) restricted on the split-complex plane is a timelike minimal surface (see Proposition \ref{sum of two maximal surfaces}).  As a consequence of this fact, we have shown that the topology of the singularity set of the timelike minimal surface depends on the singularity set of the generating maximal surfaces (see Proposition \ref{Compactness_singularity}). Furthermore, we have also studied the asymptotic behaviour of the simple ends of a timelike minimal surface (see Theorem \ref{Asymptotic behaviour}). 

	This paper is organised as follows: In section \ref{Preliminaries}, we summarize basic facts about bicomplex numbers and bicomplex functions in a form suitable for our purpose and give basic definitions, and introduce the notion of \textit{bicomplex surface}. In section \ref{section: Timelike minimal surface from bicomplex surface}, we have defined and solved the \textit{period condition} for the bicomplex surface. In section \ref{section: Singularities of timelike minimal surface}, we study the topology of the singularity set of the timelike minimal surface. In section \ref{section: Asymptotic behavior of simple ends}, we study the asymptotic behaviour of simple ends of the timelike minimal surface. Finally, in section \ref{section: arbitrary number of ends}, we prove the existence of a timelike minimal surface with $n$ number of weak complete ends.

	\section{Preliminaries} \label{Preliminaries}
 The ring of \textit{bicomplex} numbers denoted by $\mathbb{BC}$ can be identified as $\mathbb{R}^4$ with the standard topology. It is defined as
$$
\begin{aligned}
\mathbb{BC} := & \left\{\tilde{z}=x_1+ \mathbf{i} \,x_2+ \mathbf{j}\, x_3+ \mathbf{k} \,x_4 \mid x_1, x_2, x_3, x_4 \in \mathbb{R}\right\}\\
= & \left\{\tilde{z}=z_1+ \mathbf{j}\, z_2 \mid z_1=x_1+\mathbf{i}\,x_2, z_2=x_3+\mathbf{i}\,x_4 \in \mathbb{C}(\mathbf{i})\right\}, 
\end{aligned}
$$
where $\mathbf{i}$, $\mathbf{j}$ and $\mathbf{k}$ are imaginary units, satisfying,
$$
\mathbf{i} \neq \mathbf{j} , \quad \mathbf{i j}=\mathbf{j i},\quad \mathbf{k} =\mathbf{i j}, \quad \mathbf{i}^2=\mathbf{j}^2=-1,
$$
so that $\mathbf{k}^2=1$. The ring of split-complex numbers denoted by $\mathbb{D}$  is defined as,
$$\mathbb{D}:=\left\{ \hat{z}=x_1+\mathbf{k}\,x_4  \mid x_1,x_4 \in \mathbb{R}\right\} \subset \mathbb{BC}.$$

 Throughout this article, we will use the following notations and conventions.
\subsection{Notations and Conventions}

\begin{itemize}
    \item  If \(\tilde{z}=x_1+ \mathbf{i}\, x_2+ \mathbf{j} \,x_3+ \,\mathbf{k} \,x_4\), then \(\operatorname{Re}(\tilde{z}):=x_1\), \(\operatorname{Im}_{\mathbf{i}}(\tilde{z}):=x_2\), \(\operatorname{Im}_{\mathbf{j}}(\tilde{z}):=x_3\), \(\operatorname{Im}_{\mathbf{k}}(\tilde{z}):=x_4\), and $\tilde{z}$ restricted to $\mathbb{D}$ is denoted by $\tilde{z}\bigr|_{\mathbb{D}}=x_1+ \mathbf{k}\, x_4$, and $\tilde{f}: \tilde{\Omega} \subset \mathbb{BC} \rightarrow \mathbb{BC}$ denotes a bicomplex-valued function, where $\tilde{\Omega}$ is an open set in $\mathbb{BC}$, and $\tilde{f}\bigr|_{\mathbb{D}}$ denotes $\tilde{f}$ restricted to $\tilde{\Omega} \cap  \mathbb{D}$.
 %and $\tilde{f}$ is said to be of class $C^r$ if it is a $C^r$  function from an open subset of $\mathbb{R}^4$ to $\mathbb{R}^4$.

  \item 
  If \(\hat{z}=x+\mathbf{k}\,y\), then \(\operatorname{Re}(\hat{z}):=x\) and \(\operatorname{Im}(\hat{z}):=y\), 
 and $\hat{f}: \hat{\Omega} \subset \mathbb{D} \rightarrow \mathbb{D}$ denotes a split-complex-valued function, where $\hat{\Omega}$ is an open set in $\mathbb{D}$. If $\hat{z} \in \mathbb{D}$, then its norm is defined by ${|\hat{z}|}^2_{\mathbb{D}}:={x}^2-{y}^2$.

    \item The set \(\mathbb{C}(\mathbf{i})\) denotes the field complex numbers with $\mathbf{i}$ as its imaginary unit. If \ \(z= a+\mathbf{i}\,b\), then \(\operatorname{Re}(z):=a\), \(\operatorname{Im}(z):=b\), and ${f}: {\Omega} \subset \mathbb{C}(\mathbf{i}) \rightarrow \mathbb{C}(\mathbf{i})$ denotes a complex valued function, where $\Omega$ is an open set in $\mathbb{C}(\mathbf{i})$. 
    
\end{itemize}

The rings $\mathbb{B C}$ and $\mathbb{D}$ have zero divisors; the set of zero divisors in $\mathbb {BC}$ is denoted by
$
\mathfrak{S}:=\left\{\tilde{z} \mid \tilde{z} \neq 0, z_1^2+z_2^2=0\right\}
$
and let $
\mathfrak{S}_0:=\mathfrak{S} \cup\{0\}
$. Then the set $\mathfrak{S} \cap \mathbb{D}=\{\hat{z}=x+\mathbf{k}\,y \mid \hat{z} \neq 0, x^2-y^2=0\} $ will denote the set of zero divisors in $\mathbb{D}$ and let $S:=\mathfrak{S}_0 \cap \mathbb{D}$.

Let $\mathbf{e}, \mathbf{e}^{\dagger}\in\mathbb{D}$ such that
\begin{equation}
\label{def_of_e}
    \mathbf{e}:=\frac{1+\mathbf{k}}{2},\quad
 \mathbf{e}^{\dagger}:=\frac{1-\mathbf{k}}{2}.
\end{equation}
 
Then it can be easily checked that they satisfy the following:
$$
\begin{gathered}
\mathbf{e}+\mathbf{e}^{\dagger}=1 ; \quad \mathbf{e}-\mathbf{e}^{\dagger}=\mathbf{k} ; \\
\mathbf{e e}^{\dagger}=0 ; \quad \mathbf{e e}=\mathbf{e} \quad \text { and } \quad \mathbf{e}^{\dagger} \mathbf{e}^{\dagger}=\mathbf{e}^{\dagger} .
\end{gathered}
$$

Next, if we consider the following two subsets $
\mathbb{B C}_{\mathbf{e}}:=\left\{\beta_1 \mathbf{e} \mid \beta_1 \in \mathbb{C}(\mathbf{i})\right\} \text { and } \mathbb{B C}_{\mathbf{e}^{\dagger}}:=\left\{\beta_2 \mathbf{e}^{\dagger} \mid \beta_2 \in \mathbb{C}(\mathbf{i})\right\} 
$ of $\mathbb{BC}$.
Then it turns out that every $\tilde{z}$ in $\mathbb {BC}$ can be uniquely expressed as
$
\tilde{z}=\beta_1 \mathbf{e}+\beta_2 \mathbf{e}^{\dagger},
$
where
$
\beta_1=z_1- \mathbf{i}\, z_2 \text { and } \beta_2=z_1+ \mathbf{i}\, z_2
$
belongs to $\mathbb {C(\mathbf i)}$. This representation of the bicomplex number is called the idempotent representation. Furthermore, any $\hat z \in \mathbb D$ can be written as 
\begin{equation}
\label{subham_3}
    \hat{z}=\tilde{z}\bigr|_{\mathbb{D}}=\operatorname{Re}(\beta_1) \mathbf{e}+\operatorname{Re}(\beta_2) \mathbf{e}^{\dagger}.
\end{equation}
Next, for any $\tilde{z}\in\mathbb{BC}$ we denote its conjugate by $\tilde{z}^*$  and it is defined as 
 $\tilde{z}^*:=\bar{z}_1- \mathbf{j} \,\bar{z}_2=\bar{\beta}_1 \mathbf{e}+\bar{\beta}_2 \mathbf{e}^{\dagger}$, where $\bar{z}$ denotes the usual complex conjugation.

Now, we state some facts about functions defined on $\mathbb {BC}$, which will be used in this article.

Let $\tilde{F}: \tilde{\Omega} \subset \mathbb{BC} \rightarrow \mathbb{BC}$ be a bicomplex function. Then $\tilde{F}$ can be expressed in the idempotent form as:
$\tilde{F}=G_1 \mathbf{e}+G_2 \mathbf{e}^{\dagger},
$
where $G_1,G_2$ both are complex valued functions of two complex variables $\beta_1,\beta_2$ in general. Then we have the following theorems:

\begin{theorem}\label{differential thm}
(\cite{Luna2015}) The bicomplex function $\tilde{F}$ is holomorphic if and only if $G_1(\beta_1)$ and $G_2(\beta_2)$  are holomorphic functions.

\end{theorem}

 \begin{theorem}(\cite{Charak})
 The  bicomplex function $\tilde{F}$ is meromorphic if and only if  $G_1(\beta_1)$ and $G_2(\beta_2)$ are meromorphic functions.
     
 \end{theorem}
 We say $\tilde{F}$ has a pole at $\tilde{p}=p_1 \mathbf{e}+p_2 \mathbf{e}^{\dagger}$ if $G_1$ and $G_2$ has a pole at $p_1$ and $p_2$ respectively.
For more details about bicomplex functions and their properties, one can see \cite{Luna2015},\cite{Charak}, and \cite{Luna}.

 \begin{definition}
 
  Let $\hat{f}:\hat{\Omega} \subset \mathbb{D} \rightarrow \mathbb{D}$ be a function defined by $\hat{f}(\hat{z})=\hat{f}(x+\mathbf{k}\,y)=u(x,y)+\mathbf{k}\,v(x,y)$, where $u,v$ are real-valued $C^1$ functions defined on $\hat{\Omega}$. We say $\hat{f}$ is split-holomorphic if it satisfy $u_x=v_y$ and $u_y=v_x$ on $\hat{\Omega}$.
  \end{definition}
  
The question of split-meromorphic functions is more complicated. We will say that
$\hat{f}$ has a pole at $\hat{z}_0$ if and only if $\frac{1}{\hat{f}}$ is zero at $\hat{z}_0$. However, it may happen that zeroes of $\hat{f}$ are not isolated, or even if zeroes of $\hat{f}$ are isolated, then the pole is not isolated, as can be seen in the example below,
 \begin{example}\label{pole example}
Consider the following function:
 $$
 \hat{f}(\hat{z})=\frac{1}{\hat{z}}=\frac{x-\mathbf{k}y}{(x+\mathbf{k}y)(x-\mathbf{k}y)}=\frac{x-\mathbf{k}y}{x^2-y^2};\,\,\forall \hat{z}=x+\mathbf{k}y\in \mathbb{D}\backslash S.
 $$

 \end{example}

     A map $\hat{f}:\hat{\Omega} \subset \mathbb{D} \rightarrow \mathbb{D}$ is \textit{split-analytic} if it has a power series expansion in $\hat{\Omega}$. The split-analyticity implies split-holomorphicity. But the converse is not always true (see \cite{Inoguchi}).

\begin{definition}
   Let $\hat{f}:\hat{\Omega} \subset \mathbb{D} \rightarrow \mathbb{D}$ be split-analytic and let $\hat{z}_0$ be a zero of $\hat{f}$, then $\hat{f}(\hat{z})=(\hat{z}-\hat{z}_0)^m\hat{h}(\hat{z})$ in a neighbourhood of $\hat{z}_0$ for some split-analytic function $\hat{h}$ with $\hat{h}(\hat{z}_0)\neq0$. We call $m$ to be the order of zero of $\hat{f}$ at $\hat{z}_0$. We would call a function $\hat{g}$ to be split-meromorphic if $\hat{g}=\frac{1}{\hat{f}}$, where $\hat{f}$ is a split-analytic function and $\hat{p}$ is called a pole of $\hat{g}$ if $\hat{p}$ is a zero of $\hat{f}$. 
\end{definition}
Let $\hat{P}$ denote the set of poles of the split-meromorphic function $\hat{g}$, then we call $\hat{P}+S$ the set of associated zero divisors of $\hat{P}$. 
\begin{remark}\label{remak for pole of split-meromorphic}
 In order for the function $\hat{g}$ (as above) to be well defined on $\hat{\Omega}$, we have to remove the $\hat{P}+S$ from $\hat{\Omega}$ (see Example \ref{pole example}). Now, note that the domain $\hat{\Omega}\backslash (\hat{P}+S)$ is disconnected.
\end{remark}
 
In this article, by split-holomorphic function, we always mean split-analytic function only. For more details about the theory of split-holomorphic and split-meromorphic functions on $\mathbb{D}$, one can check \cite{Inoguchi} and \cite{Akamine}.

Next, we state some facts about the bicomplex functions, which highlight their relation to the split-holomorphic and split-meromorphic functions.

\begin{fact}(\cite{Kato})
    Let $\tilde{F}: \tilde{\Omega} \rightarrow \mathbb{BC}$ be a bicomplex function satisfying 
\begin{equation}\label{5.4}
    \tilde{F}(\tilde{z}^*)=(\tilde{F}(\tilde{z}))^*
    \quad \forall \tilde{z} \in \tilde{\Omega} \cap \tilde{\Omega}^* \subset \mathbb{BC}. 
\end{equation}
Then $\tilde{F}(\tilde{z}) \in \mathbb{D} \text{ for all }  \tilde{z} \in \tilde{\Omega} \cap  \mathbb{D} $. In other words, $\tilde{F}\bigr|_{\mathbb{D}}$  becomes a split-complex valued function.
\end{fact}
As a consequence of this fact, we get 
\begin{fact}\label{split-holo and split-mero}
Let \(\tilde{F}: \tilde{\Omega} \rightarrow \mathbb{BC}\) be a bicomplex function that satisfies equation \eqref{5.4}. If \(\tilde{F}\) is bicomplex holomorphic (respectively, bicomplex meromorphic) on \(\tilde{\Omega} \subset \mathbb{BC}\), then \(\tilde{F}\bigr|_{\mathbb{D}}\) is split-holomorphic (respectively, split-meromorphic).
\end{fact}
Now, we will recall the Weierstrass-Enneper representation for a maximal surface, which can be stated as follows:
\begin{theorem}\label{thm maximal}(\cite{Romero})
    Let $\Omega$ be an open subset of $\mathbb{C}(\mathbf{i})$, $g$ be a meromorphic function, and $f$ be a holomorphic function on $\Omega$ such that :
\begin{enumerate}
    \item $fg^2$ is holomorphic on $\Omega$.
    \item $\operatorname{Im}(g) \not\equiv 0 $ on $\Omega$,
    \item $ \operatorname{Re}\int_{\gamma} \left(\left(1-{g}^2\right),2g,\left(1+{g}^2\right)\right)f\,dz=0 $ for all closed loops $\gamma$ on $\Omega$.  
\end{enumerate}
    Then, the map $X:\Omega \rightarrow \mathbb{E}_1^3$ given by 
    \begin{equation}\label{maximal period condition}
        X(z)=  \operatorname{Re}\int_{z_0}^{z} \left(\left(1-g^2\right),2g,\left(1+g^2\right)\right) f\,dz
    \end{equation}
    defines a maximal surface with base point $z_0\in \Omega$, where $(g, f)$ is called the Weierstrass data of $X$.

\end{theorem}
The third condition in Theorem \ref{thm maximal} is referred to as the \textit{period condition} for the maximal surface. The induced metric on $X$ coming from $\mathbb{E}^3_1$ is given by $h =4\left(\operatorname{Im}(g)\right)^2| f|^2$ and the set of points where the metric $h$ degenerates is $\{z \in \Omega: \operatorname{Im}g(z) = 0 \text{ or } f(z)=0\}$. The set $\operatorname{Sing}(X):=\{z \in \Omega:\operatorname{Im}g(z) = 0\}$ denotes the singularity set of $X$ and $\{z \in \Omega: f(z)=0\}$ is called the branch point set of $X$.

Now, we introduce the timelike minimal surface in $\mathbb{E}^3_1$:
\begin{definition}(\cite{Kim})
   Let $U$ be an open set in $\mathbb{R}^2$, and let $X:U \subset \mathbb{R}^2 \rightarrow \mathbb{E}^3_1$ be a $C^2$ map with the property that 
    \begin{equation}\label{conformal}
        \langle X_u,X_u \rangle =- \langle X_v,X_v \rangle, \quad \langle X_u,X_v \rangle =0.
    \end{equation}
    Then $X(u,v)$ is a timelike minimal surface if and only if 
    \begin{equation}\label{wave}
        X_{uu}-X_{vv}=0.
    \end{equation}
\end{definition}

 Similar to the case of minimal and maximal surfaces, Jerzy J. Konderak \cite{Konderak} has obtained the Weierstrass-Enneper type representation for timelike minimal surfaces using the theory of split-complex numbers. 
 \begin{theorem}\label{thm timelike}
    Let $\hat{\Omega}$ be an open subset of $\mathbb{D}$, $\hat{g}$ be a split-meromorphic function, and $\hat{f}$ be a split-holomorphic function on $\hat{\Omega}$ such that the following hold:
\begin{enumerate}
    \item  $\hat{f}\hat{g}^2$ is split-holomorphic on $\hat{\Omega}$,
    \item $\operatorname{Im}(\hat{g}) \not\equiv 0 $ on $\hat{\Omega}$,
    \item $ \operatorname{Re}\int_{\hat{\gamma}} \left(\left(1-{\hat{g}}^2\right),2\hat{g},\left(1+{\hat{g}}^2\right)\right)\hat{f}\,d\hat{z}=0 $ for all closed loops $\hat{\gamma}$ on $\hat{\Omega}$.
    \end{enumerate}Then 
    \begin{equation}{\label{timelike}}
        \hat{X}(\hat{z})=  \operatorname{Re}\int_{\hat{z}_0}^{\hat{z}} \left(\left(1-{\hat{g}}^2\right),2\hat{g},\left(1+{\hat{g}}^2\right)\right)\hat{f}\,d\hat{z};
    \end{equation}
    defines a timelike minimal surface with base point $\hat{z}_0\in \hat{\Omega}$, where $(\hat{g}, \hat{f})$ is called the Weierstrass data of $\hat{X}$.
\end{theorem}
      The third condition in Theorem \ref{thm timelike} is called the \textit{period condition} for the timelike minimal surface $\hat{X}$. The induced metric on $\hat{X}$ coming from $\mathbb{E}^3_1$ is given by $\hat{h}=-4\left(\operatorname{Im}(\hat{g})\right)^2|\hat{f}|_{\mathbb{D}}^2$. The set of points where the metric $\hat{h}$ degenerates is $\{\hat{z} \in \hat{\Omega}: \operatorname{Im}\hat{g}(\hat{z}) = 0 \text{ or } \hat{f}(\hat{z})=0\}$. The set $\operatorname{Sing}(\hat{X}):=\{\hat{z} \in \hat{\Omega}:\operatorname{Im}\hat{g}(\hat{z}) = 0\}$, denotes the singularity set of $\hat{X}$ and $\{\hat{z} \in \hat{\Omega}: \hat{f}(\hat{z})=0\}$ the branch point set of $\hat{X}$. The induced metric on $\hat{X}$ coming from $\mathbb{E}^3$ is given by
 \begin{equation}\label{weak complete}  \hat{h}_{\mathbb{E}^3}=2\left(1+|\hat{g}|_{\mathbb{D}}^2\right)^2|\hat{f}|_{\mathbb{D}}^2|d {\hat{z}}|_{\mathbb{D}}^2 + \operatorname{ Re} \{\left(1+\hat{g}^2\right)^2\hat{f}^2 d {\hat{z}}^2\}.
     \end{equation}
\begin{definition}\label{def weak complete}
    We say a timelike minimal surface $\hat{X}$ is $\mathbb{E}^3$-complete\footnote{Senchun Lin and Tilla Weinstein first used this term in \cite{lin19973}.} or weak complete\footnote{This term was first used in the context of maximal surface by Masaaki Umehara and Kotaro Yamada in \cite{umehara2006maximal}.} if it is complete with respect to the metric $\hat{h}_{\mathbb{E}^3}$.

\end{definition}
For more details about timelike minimal surfaces, one can check  \cite{Inoguchi} and \cite{Erdem}. 

Now, we define the end of a timelike minimal surface similar to how Taishi Imaizumi defined it for a maximal surface in \cite{Imaizumi}. 

\begin{definition}\label{End definition}
   Let $\hat{\Omega} \subset \mathbb{D}$ be an open set, $\hat{g}$ be a split-meromorphic function, and $\hat{f}$ be a split-holomorphic function defined on $\hat{\Omega}$. Also, let $\hat{P}$ be the set of poles of $\hat{\alpha}=\left(\left(1-{\hat{g}}^2\right),2\hat{g},\left(1+{\hat{g}}^2\right)\right)\hat{f}$. Now, if the map $\hat{X}: \hat{\Omega}^*= \hat{\Omega} \backslash (\hat{P}+S) \rightarrow \mathbb{E}^3_1$, defined by equation \eqref{timelike} is a timelike minimal surface with Weierstrass data $(\hat{g}\bigr|_{\hat{\Omega}^*}, \hat{f}\bigr|_{\hat{\Omega}^*})$, then we define a point $\hat{p} \in \hat{P}$ to be an end of $\hat{X}$. Also, we say $\hat{p}$ is a weak complete end if $\hat{X}$ is weak complete in a neighbourhood around $\hat{p}$.
\end{definition}
Here, note that the ends of a timelike minimal surface are quite different from the ends of a maximal surface and a minimal surface.
 As we can see from Remark \ref{remak for pole of split-meromorphic}, if $\hat{X}$ has ends, then defining the period condition on $\hat{\Omega}\backslash (\hat{P}+S)$ is difficult. This makes studying timelike minimal surfaces with ends nontrivial. We have addressed this problem in section \ref{section: Timelike minimal surface from bicomplex surface} using the theory of bicomplex numbers.

Kaname Hashimoto and Shin Kato were the first to present the bicomplex extension of zero mean curvature surfaces in $\mathbb{E}^3_1$ (see section $4$ of \cite{Kato}). For simplicity, we will use the term ``bicomplex surface" instead of ``bicomplex extension of zero mean curvature surface." Their representation unfolds as follows:

Any bicomplex surface can be locally represented as 

 \begin{equation}\label{eq For BC}
\tilde{\Phi}(\tilde{z})= \int^{\tilde{z}}_{\tilde{z}_0} \left(\left(1-\tilde{g}^2\right),2\tilde{g},\left(1+\tilde{g}^2\right)\right) \tilde{f}\,d\tilde{z}
\end{equation}
where, $\tilde{f}$ is a bicomplex holomorphic
function and $\tilde{g}$ is a bicomplex meromorphic function. 
Suppose that $(\tilde{g},\tilde{f})$ satisfies the equation \eqref{5.4}, then using the fact $\ref{split-holo and split-mero}$ we get the projection $\operatorname{Re}{\tilde{\Phi}\bigr|_{\mathbb{D}}}$ is minimal timelike immersion from $\tilde{\Omega} \cap \mathbb{D}$ into $\mathbb{E}^3_1$ with Weierstrass data ($\tilde{g}\bigr|_{\mathbb{D}},\tilde{f}\bigr|_{\mathbb{D}}$).

    \begin{remark}\label{split-harmonic}(\cite{Kato})
     The $\operatorname{Re}{\tilde{\Phi}\bigr|_{\mathbb{D}}}$ is a split-harmonic map
with respect to the metric $dx_1^2-dx_4^2$
on $\tilde{\Omega} \cap \mathbb{D}$.
\end{remark}
A complex-valued function \( G \) of a complex variable \( \beta \) is called \textit{conjugate-symmetric} if it satisfies the condition \( G(\bar{\beta}) = \overline{G(\beta)} \).
\begin{remark}\label{conjugate-symmetric}
 A bicomplex function \( \tilde{h}(\tilde{z}) = G_1(\beta_1)\mathbf{e} + G_2(\beta_2)\mathbf{e}^\dagger \) satisfies equation \eqref{5.4} if and only if both \( G_1 \) and \( G_2 \) are conjugate-symmetric.
\end{remark}

    \section{ Timelike minimal surface from bicomplex surface} \label{section: Timelike minimal surface from bicomplex surface}
    In \cite{McNertney}, it is shown that any timelike minimal surface can be locally expressed as a sum of two null curves. Motivated by this idea, we aim to construct a bicomplex surface as a sum of two maximal surfaces, as detailed in this section. Furthermore, by restricting the domain of a bicomplex surface to the split-complex plane, we obtain a timelike minimal surface.

The following lemma can be easily verified:

\begin{lemma}\label{lemma of BC}
    Let $\tilde{z} = \beta_1 \mathbf{e} + \beta_2 \mathbf{e}^\dagger \in \mathbb{BC}$, then $\operatorname{Re} \tilde{z} = \frac{1}{2} (\operatorname{Re} \beta_1 + \operatorname{Re} \beta_2)$.
\end{lemma}

\begin{remark}\label{lemma of curve}
    If $\tilde{\gamma}:[a,b] \rightarrow \mathbb{BC}$ be a curve, then $\tilde{\gamma}(t)=\gamma_1(t) \mathbf{e} + \gamma_2(t) \mathbf{e}^\dagger;  \forall t\in [a,b]$, where $\gamma_1(t) $ and $\gamma_2(t)$ are curves in $\mathbb{C}(\mathbf{i})$. Furthermore, $\tilde{\gamma}$ is a loop  if and only if $\gamma_1(t) $ and $\gamma_2(t)$ are loops.
\end{remark}

    %\subsection{Real period condition for bicomplex surface}
  %================================

Let $\tilde {g}$ be a bicomplex meromorphic function and $\tilde {f}$ be a bicomplex holomorphic function defined on $\tilde {\Omega} \subset \mathbb{BC}$. If for all loops $\tilde{\gamma}$ in $\tilde{\Omega} \backslash \tilde{P}$
\begin{equation}
\label{subh01}
    \operatorname{Re}\int_{\tilde{\gamma}} \left(\left(1-\tilde{g}^2\right), 2\tilde{g}, \left(1+\tilde{g}^2\right)\right) \tilde{f} \, d\tilde{z} = 0
\end{equation}
 where $\tilde{P}$ is the set of poles of $\left(\left(1-\tilde{g}^2\right), 2\tilde{g}, \left(1+\tilde{g}^2\right)\right) \tilde{f}$.  We call equation \eqref{subh01} the ``real period condition for bicomplex surfaces''. Furthermore, if $(\tilde{g}, \tilde{f})$ satisfies the condition given in the equation  \eqref{5.4}, then using the Fact $\ref{split-holo and split-mero}$ and the equation \eqref{eq For BC}, $\hat{X} := \operatorname{Re}(\tilde{\Phi}(\tilde{z}))\bigr|_{\mathbb{D}}$ gives a timelike minimal surface on $(\tilde{\Omega} \backslash \tilde{P}) \cap \mathbb{D}$ with the Weierstrass data given by $(\tilde{g}\bigr|_{\mathbb{D}}, \tilde{f}\bigr|_{\mathbb{D}})$.

 \begin{remark}
  It is important to note that a timelike minimal surface can be obtained without requiring \((\tilde{g}, \tilde{f})\) to satisfy the condition specified in equation \eqref{5.4}. By utilizing Remark \ref{split-harmonic}, we observe that all three components of \(\operatorname{Re}(\tilde{\Phi}(\tilde{z}))\bigr|_{\mathbb{D}}\) are split-harmonic, i.e., it satisfies equation \eqref{wave}. Also, the parametrization \(\tilde{\Phi}\) satisfies equation \eqref{conformal}. As a result, \(\hat{X} = \operatorname{Re}(\tilde{\Phi}(\tilde{z}))\bigr|_{\mathbb{D}}\) defines a timelike minimal surface on \((\tilde{\Omega} \backslash \tilde{P}) \cap \mathbb{D}\). However, this approach causes the loss of the Weierstrass data for our timelike minimal surface because now \((\tilde{g}\bigr|_{\mathbb{D}}, \tilde{f}\bigr|_{\mathbb{D}})\) are surfaces from \(\mathbb{D}\) to \(\mathbb{BC}\).

 \end{remark}
Next, we find a suitable \((\tilde{g}, \tilde{f})\) such that the real period condition given in equation \eqref{subh01} satisfies.
 %\subsection{Solution of real period condition for bicomplex surface via period condition of maximal surfaces}

Using the period condition for maximal surface, we solve the real period condition for bicomplex surfaces given in equation \eqref{subh01}. We assume $X_k: \Omega_k \subset \mathbb{C}(\mathbf{i}) \rightarrow \mathbb{E}_1^3$ are maximal surfaces with Weierstrass data $(g_k, f_k) (k=1, 2)$ that satisfy the period condition for maximal surfaces given in the equation \eqref{maximal period condition}. To define a bicomplex surface, we take $\tilde{g} := g_1 \mathbf{e} + g_2 \mathbf{e}^{\dagger}$ as a bicomplex meromorphic function and $\tilde{f} := f_1 \mathbf{e} + f_2 \mathbf{e}^{\dagger}$ as a bicomplex holomorphic function on $\tilde{\Omega} = \Omega_1 \mathbf{e} + \Omega_2 \mathbf{e}^{\dagger} \subset \mathbb{BC}$. Now, if we substitute the idempotent expressions of $\tilde{g}$ and $\tilde{f}$ in equation \eqref{subh01} and then using the Lemma \ref{lemma of BC} and Remark \ref{lemma of curve} along with the properties of $\mathbf{e},\,\mathbf{e}^{\dagger}$, as mentioned earlier (see equation \eqref{def_of_e}), the real period condition of bicomplex surface splits into period condition of maximal surfaces $X_1$ and $X_2$ as follows:

\begin{eqnarray*}
\lefteqn{ \operatorname{Re}\int_{\tilde{\gamma}} \left(\left(1-\tilde{g}^2\right), 2\tilde{g}, \left(1+\tilde{g}^2\right)\right) \tilde{f} \, d\tilde{z} = } \\
& & \frac{1}{2}\left\{\operatorname{Re}\int_{\gamma_1} \left(\left(1-g_1^2\right), 2g_1, \left(1+g_1^2\right)\right) f_1 \, d\beta_1 + \operatorname{Re}\int_{\gamma_2} \left(\left(1-g_2^2\right), 2g_2, \left(1+g_2^2\right)\right) f_2 \, d\beta_2\right\} = 0,
\end{eqnarray*}
for any loop $\tilde{\gamma} = \gamma_1 \mathbf{e} + \gamma_2 \mathbf{e}^{\dagger}$ in $\Omega_1 \mathbf{e} + \Omega_2 \mathbf{e}^{\dagger} \subset \mathbb{BC}$.

This shows that $\operatorname{Re}(\tilde{\Phi})$ (see equation \eqref{eq For BC}) is well-defined even if $\tilde{\Phi}$ is not well-defined, and $\operatorname{Re}{\tilde{\Phi}\bigr|_{\mathbb{D}}}$ is a timelike minimal surface. Moreover, if $(g_k, f_k)$ (for $k = {1, 2}$) are conjugate-symmetric as mentioned in Remark $\ref{conjugate-symmetric}$, then the Weierstrass data of the timelike minimal surface is given by $(\tilde{g}\bigr|_{\mathbb{D}}, \tilde{f}\bigr|_{\mathbb{D}})$.

Now, we can state the following proposition:
\begin{prop} \label{sum of two maximal surfaces}
    Let $X_1$ and $X_2$ be the maximal surfaces with Weierstrass data $(g_1,f_1)$ and $(g_2,f_2)$, and $\tilde{g}:= g_1 \mathbf{e} + g_2 \mathbf{e}^{\dagger}$ and $\tilde{f} := f_1 \mathbf{e} + f_2 \mathbf{e}^{\dagger}$  both are satisfying \eqref{5.4}. Then,
$\hat{X}=\frac{1}{2}(X_1+X_2) \bigr|_{\mathbb{D}}$ is a timelike minimal surface with the Weierstrass data $(\tilde{g}\bigr|_{\mathbb{D}},\tilde{f}\bigr|_{\mathbb{D}})$.
\end{prop}

\begin{example} \label{example}
    Let \( g_1(\beta_1) = -\beta_1 \), \( g_2(\beta_2) = -\beta_2 \), \( f_1(\beta_1) = -\frac{1}{{\beta_1}^2} \), and \( f_2(\beta_2) = -\frac{1}{{\beta_2}^2} \). From equation \eqref{maximal period condition}, we obtain 
    \[
    X_1 = \left( \frac{a_1}{a_1^2 + b_1^2} + a_1, \ln(a_1^2 + b_1^2), \frac{a_1}{a_1^2 + b_1^2} - a_1 \right),
    \]
    and 
    \[
    X_2 = \left( \frac{a_2}{a_2^2 + b_2^2} + a_2, \ln(a_2^2 + b_2^2), \frac{a_2}{a_2^2 + b_2^2} - a_2 \right),
    \]
    where \( \beta_1 = a_1 + \mathbf{i} b_1 = (x_1 + x_4) + \mathbf{i} (x_2 - x_3) \) and \( \beta_2 = a_2 + \mathbf{i} b_2 = (x_1 - x_4) + \mathbf{i} (x_2 + x_3)  \).

    Now, using Proposition \ref{sum of two maximal surfaces}, we compute 
    \[
    \hat{X} = \frac{1}{2} (X_1 + X_2) \Big|_{\mathbb{D}} = \frac{1}{2} \left( \frac{1}{a_1} + a_1 + \frac{1}{a_2} + a_2, \ln(a_1^2) + \ln(a_2^2), \frac{1}{a_1} - a_1 + \frac{1}{a_2} - a_2 \right).
    \]
    Further simplification yields
    \[
    \hat{X}(x_1, x_4) = \left( \frac{x_1}{x_1^2 - x_4^2} + x_1, \ln\left| x_1^2 - x_4^2 \right|, \frac{x_1}{x_1^2 - x_4^2} - x_1 \right),
    \]
    which defines a timelike minimal surface on \( \mathbb{D} \backslash S \).
\end{example}

  %================================

\section {Singularities of timelike minimal surfaces}\label{section: Singularities of timelike minimal surface}
Let $X_1$ and $X_2$ be two maximal surfaces as introduced in section \ref{section: Timelike minimal surface from bicomplex surface} with their Weierstrass data $(g_1,\,f_1)$ and $(g_2,\,f_2)$ respectively such that the functions $g_k,\,f_k (k=1,2)$ are conjugate symmetric. Therefore, $\tilde{g}=g_1\mathbf{e}+g_2\mathbf{e}^{\dagger}$ is a bicomplex meromorphic function and  $\tilde{f}=f_1\mathbf{e}+f_2\mathbf{e}^{\dagger}$ is a bicomplex holomorphic function 
defined on $\tilde{\Omega}=\Omega_1 \mathbf{e} + \Omega_2 \mathbf{e}^{\dagger} \subset \mathbb{BC}$ and they satisfy the equation \eqref{5.4}. So, $\hat{X}= \frac{1}{2}(X_1+X_2) \bigr|_{\mathbb{D}}$ becomes a timelike minimal surface on $\hat{\Omega}=\tilde{\Omega} \cap \mathbb{D}$ with Weierstrass data $(\tilde{g}\bigr|_{\mathbb{D}} , \tilde{f}\bigr|_{\mathbb{D}})$ (see Proposition \ref{sum of two maximal surfaces}). Also, if $\operatorname{Sing}({X}_1)$ and $\operatorname{Sing}({X}_2)$ are the sets of all singular points of $X_1$ and $X_2$, respectively, then we have the following lemma. 
\begin{lemma}
\label{newlemma}
    $\hat{\Omega}\subset\operatorname{Sing}({X}_1) \mathbf{e} + \operatorname{Sing}({X}_2) \mathbf{e}^\dagger$.
\end{lemma}
\begin{proof}
    The functions $g_1, g_2$ are conjugate symmetric. If $\hat{z}= a\mathbf{e} + b\mathbf{e}^\dagger\in\hat{\Omega}$, then from equation \eqref{subham_3} it follows that $a,\,b$ are real. This implies $g_1(a)=\overline{g_1(a)}$ and $g_2(b)=\overline{g_2(b)}$. Therefore, $\operatorname{Im}(g_1(a))=0$ and $\operatorname{Im}(g_2(b))=0$. Hence, the claim.
\end{proof}

\begin{remark}
   If \,$\operatorname{Sing}({X}_1)\subset\mathbb{R}$ and\, $\operatorname{Sing}({X}_2)\subset\mathbb{R}$, then $\hat{\Omega}= \operatorname{Sing}({X}_1) \mathbf{e} + \operatorname{Sing}({X}_2) \mathbf{e}^\dagger$. 
\end{remark}
Now, recall $\operatorname{Sing}(\hat{X})= \{\hat{z}=a\mathbf{e}+b\mathbf{e}^{\dagger}\in \hat{\Omega}: \operatorname{Im}(\hat{g}) = 0\}$. Therefore, using the expressions of $\mathbf{e}, \mathbf{e}^{\dagger}$ from equation \eqref{def_of_e} and the fact that $\hat{g}=\tilde{g}\bigr|_{\mathbb{D}}$, we see $x_1+\mathbf{k}\,x_4=a\mathbf{e}+b\mathbf{e}^{\dagger}\in \operatorname{Sing}(\hat{X})$ if and only if $g_1(a)=g_2(b)$ and $a=x_1+x_4,\, b=x_1-x_4$. Hence, we write
\begin{equation}
  \operatorname{Sing}(\hat{X})= \{x_1+\mathbf{k}\,x_4\in\hat{\Omega}:g_1(x_1+x_4) = g_2(x_1-x_4)\}.  
\end{equation} 
From above it is clear that $\operatorname{Sing}(\hat{X})$ is a closed disconnected subset of $$\hat{\Omega} \backslash \left(\bigcup_{i=1}^m\{x_1+\mathbf{k}\,x_4\in\hat{\Omega}:x_1+x_4=a_i\}\cup \bigcup_{j=1}^n\{x_1+\mathbf{k}\,x_4\in\hat{\Omega}:x_1-x_4=b_j\}\right)=\chi\,(\text{say}),$$ where $a_1,\ldots, a_m$ are the poles of $g_1$ in $\Omega_1\cap\mathbb{R}$  and $b_1,\ldots, b_n$  are the poles of $g_2$ in $\Omega_2\cap\mathbb{R}$.  Now we have the following proposition:
\begin{prop}
\label{Compactness_singularity}
    If $\operatorname{Sing}({X}_1)\subset A_1 \subset \Omega_1$ and $\operatorname{Sing}({X}_2) \subset A_2 \subset \Omega_2$, where $A_1, A_2$ are compact subsets in $\mathbb{C}(\mathbf{i})$ such that $A_1$ and $A_2$ do not contain any pole of $g_1$ and $g_2$ respectively, then $\operatorname{Sing}(\hat{X})$ will be compact.
\end{prop}
\begin{proof}
   From the Lemma \ref{newlemma} and the fact $\operatorname{Sing}(\hat{X})\subset \hat{\Omega}$, we have $\operatorname{Sing}(\hat{X}) \subset (A_1 \mathbf{e}+ A_2 \mathbf{e}^\dagger) \cap \mathbb{D} \subset \chi$. Since $A_1, A_2$ are compact subsets, $(A_1 \mathbf{e}+ A_2 \mathbf{e}^\dagger) \cap \mathbb{D}$ is also compact. Hence, $\operatorname{Sing}(\hat{X})$ is compact, being a closed subset of a compact set.
\end{proof}
\begin{cor}
    If $\operatorname{Sing}(X_1)$ and $\operatorname{Sing}(X_2)$ are compact, then so is $\operatorname{Sing}(\hat{X})$.
\end{cor}
 
\begin{prop}
   $\operatorname{Sing}(\hat{X})$ is a discrete subset of $\hat{\Omega}$ or whole of $\hat{\Omega}$.
\end{prop}
\begin{proof}
 Consider the analytic function of two complex variables $F:\Omega_1 \times \Omega_2 \rightarrow \mathbb{C}$, defined by $F(\beta_1,\, \beta_2)=g_1(\beta_1)-g_2(\beta_2),\,\forall (\beta_1,\beta_2)\in\Omega_1 \times \Omega_2$, where $g_1,\,g_2$ are as before and $\hat{\Omega}=\tilde{\Omega}\cap\mathbb{D}$, $\tilde{\Omega}$ is identified to $\Omega_1 \times \Omega_2$. Note that $F\equiv 0$ gives $\operatorname{Sing}(\hat{X})=\hat{\Omega}$, and if $F\not\equiv 0$, then in any connected open subset of $\Omega_1 \times \Omega_2$, the set $\{(\beta_1,\beta_2): F(\beta_1,\beta_2)=0\}$ can not have a limit point. This translates into the fact that if $a_0 \mathbf{e}+ b_0 \mathbf{e}^{\dagger} \in \operatorname{Sing}(\hat{X})$, then there exists a sufficiently small neighbourhood of $a_0 \mathbf{e}+ b_0 \mathbf{e}^{\dagger}$ containing no other singularity of $\hat{X}$.  
\end{proof}

 \section{Asymptotic behaviour of simple ends} \label{section: Asymptotic behavior of simple ends}
 It is a well-known fact that an embedded end of a complete minimal surface in $\mathbb{E}^3$ with finite
total curvature is
asymptotic to either a plane or a half of a catenoid (see \cite{jorge1983topology}).
 Analogous to this fact, Imaizumi in \cite{Imaizumi} characterizes the asymptotic behaviour of simple end for maximal surfaces into three types. Here, we start by recalling the definition of a simple end of a maximal surface.

 \begin{definition}(\cite{Imaizumi}) A maximal surface $X$  with Weierstrass data $(g,f)$ has a simple end at point $p$ if  $\alpha=\left(\left(1-g^2\right),2g,\left(1+g^2\right)\right)f$ has a pole of order two at point $p$.

 \end{definition}

 Let $\hat{X}=\frac{1}{2}(X_1+X_2) \bigr|_{\mathbb{D}}$ be a timelike minimal surface as in Proposition \ref{sum of two maximal surfaces}. Then we have the following definition
 \begin{definition}

A timelike minimal surface $\hat{X}=\frac{1}{2}(X_1+X_2) \bigr|_{\mathbb{D}}$ is said to have a simple end at $p_1 \mathbf{e}+p_2\mathbf{e}^\dagger \in \hat{\Omega}$, if the maximal surfaces $X_1$ and $X_2$ have simple ends at $p_1$ and $p_2$, respectively.

\end{definition}

\begin{Note} 
It is evident that the above definition of simple ends aligns with the Definition \ref{End definition} of ends for timelike minimal surfaces. 
\end{Note}
 
We now present a version of Imaizumi's classification of the ends of a maximal surface that suits our purpose. This is done by requiring the Weierstrass data $ (g, f)$ to be conjugate-symmetric in Theorem 2.7 of \cite{Imaizumi}, we obtain the following lemma:

 \begin{lemma}\label{ends of maximal}
Let $\Omega= \{\beta \in \mathbb{C}(\mathbf{i}): |\beta|<1\} $and let $\Omega^*:=\Omega \backslash\{0\}$. Assume $X: \Omega^* \rightarrow \mathbb{E}_1^3$ is a maximal surface with conjugate-symmetric Weierstrass data $(g,f)$ such that the end $0$ is simple. Then $X$ in a neighbourhood of $0$ is asymptotic to one of the following types:
\begin{enumerate}
    \item $\left(a r^{-1} \cos \theta, \,c \log r,\, a r^{-1} \cos \theta\right)$,
\item $\left(a r^{-1} \cos \theta+c \log r,\, 0,\,a r^{-1} \cos \theta+c \log r\right)$,

\end{enumerate}
 where $\beta:=r e^{\mathbf{i} \theta}\in \Omega, a \in \mathbb{R} \backslash\{0\}$, and $c \in \mathbb{R}$.
     
 \end{lemma}

 Let $\beta_1, \beta_2 \in \mathbb{C}(\mathbf{i})\backslash\{0\}$. If we write $\beta_1 = r_1 (\cos{\theta_1} + \mathbf{i} \sin{\theta_1})$ and $\beta_2 = r_2 (\cos{\theta_2} + \mathbf{i} \sin{\theta_2})$, where, $r_k>0$ and $\theta_k \in (-\pi,\pi] $ for $k=1,2$; then $\tilde{z} =\beta_1 \mathbf{e} + \beta_2 \mathbf{e}^{\dagger} \in \mathbb{D}\backslash\{0\}$ if and only if $\theta_1=0 \text{ or } \pi$ and $\theta_2=0 \text{ or } \pi$.  Therefore, $\epsilon_1 r_1 \mathbf{e}+\epsilon_2 r_2\mathbf{e}^{\dagger} \in \mathbb{D}\backslash\{0\}$, where $\epsilon_k\in\{1,-1\}$ for $k=1,2$, and here we say $\tilde{f}\bigr|_{\mathbb{D}}(\beta_1 \mathbf{e} + \beta_2 \mathbf{e}^{\dagger})= \tilde{f}(\epsilon_1 r_1 \mathbf{e}+\epsilon_2 r_2\mathbf{e}^{\dagger})$. 
 %$$\hat{z} = \tilde{z} \bigr|_{\mathbb{D}}= r_1 (\cos{\theta_1}+\mathbf{i} \sin{\theta_1}) \mathbf{e}+r_2 (\cos{\theta_2}+\mathbf{i} \sin{\theta_2}) \mathbf{e}^{\dagger}\bigr|_{\theta_1=0,\pi;\theta_2=0,\pi}$$ 

Now, the behaviour of an end of $\hat{X}$ depends on the behaviour of $X_1$ and $X_2$ in a neighbourhood of $0$. Let $\hat{\Omega}=(\Omega_1 \mathbf{e}+\Omega_2 \mathbf{e}^\dagger) \cap \mathbb{D}$, where $\Omega_1= \{\beta_1 \in \mathbb{C}(\mathbf{i}): |\beta_1|<1\}$ and $\Omega_2= \{\beta_2 \in \mathbb{C}(\mathbf{i}): |\beta_2|<1\}$.
Hence, we have the following result:

\begin{theorem}\label{Asymptotic behaviour}
    Let  $\hat{X}:\hat{\Omega}\backslash S \rightarrow \mathbb{E}_1^3$ be a timelike minimal surface given by $\hat{X}=\frac{1}{2}(X_1+X_2) \bigr|_{\mathbb{D}}$. Then $\hat{X}$ in a neighbourhood of  $0$ is asymptotic to one of the following types: 
    \begin{enumerate}
    \item $\hat{X}(\epsilon_1 r_1 \mathbf{e}+\epsilon_2 r_2\mathbf{e}^{\dagger})=\frac{1}{2}( 
    (\epsilon_1 a_1 r_1^{-1}+\epsilon_2 a_2 r_2^{-1} , c_1 \log r_1 + c_2 \log r_2, \epsilon_1 a_1 r_1^{-1}+\epsilon_2 a_2 r_2^{-1} ) $

   \item $\hat{X}(\epsilon_1 r_1 \mathbf{e}+\epsilon_2 r_2\mathbf{e}^{\dagger})=\frac{1}{2}\left( 
   (\epsilon_1 a_1 r_1^{-1}+\epsilon_2 a_2 r_2^{-1} +c_2 \log r_2, c_1 \log r_1, \epsilon_1 a_1 r_1^{-1}+\epsilon_2 a_2 r_2^{-1}+c_2 \log r_2 ) \right.$ 
     \item $\hat{X}(\epsilon_1 r_1 \mathbf{e}+\epsilon_2 r_2\mathbf{e}^{\dagger})=\frac{1}{2}(\epsilon_1 a_1 r_1^{-1} +c_1 \log r_1+\epsilon_2 a_2 r_2^{-1},c_2 \log r_2,\epsilon_1 a_1 r_1^{-1}+c_1 \log r_1+\epsilon_2 a_2 r_2^{-1}) $
   \item $\hat{X}(\epsilon_1 r_1 \mathbf{e}+\epsilon_2 r_2\mathbf{e}^{\dagger})=\frac{1}{2} 
   (\epsilon_1 a_1 r_1^{-1} +c_1 \log r_1+\epsilon_2 a_2 r_2^{-1} +c_2 \log r_2, 0,\epsilon_1 a_1 r_1^{-1}+c_1 \log r_1+\epsilon_2 a_2 r_2^{-1}+c_2 \log r_2)$ 
  
\end{enumerate}
 where, $\epsilon_k\in\{1,-1\}$ for $k=1,2$, $a_1,a_2 \in \mathbb{R} \backslash\{0\}$, and $c_1,c_2 \in \mathbb{R}$.
\end{theorem}

Proof of the above theorem can be seen from Proposition \ref{sum of two maximal surfaces} and Lemma \ref{ends of maximal}.\\

\begin{example}

Recall in Example \ref{example}, both \( X_1 \) and \( X_2 \) have poles at \( 0 \in \mathbb{C}(\mathbf{i}) \). Consequently, the timelike surface \( \hat{X} \) also has a pole at \( 0\mathbf{e} + 0 \mathbf{e}^{\dagger} = 0 \in \mathbb{D} \backslash S \). The asymptotic behaviour of \( X_1 \) around $0$ is given by:
\[
\left( -\frac{1}{r_1}\cos{\theta_1}, 2\ln{r_1}, -\frac{1}{r_1}\cos{\theta_1} \right),
\]
and that of \( X_2 \) is:
\[
\left( -\frac{1}{r_2}\cos{\theta_2}, 2\ln{r_2}, -\frac{1}{r_2}\cos{\theta_2} \right),
\]
Therefore, \( \hat{X} \) in a neighbourhood of $0$ is asymptotic to:
\[
\frac{1}{2}  \left( -\frac{\epsilon_1}{r_1}-\frac{\epsilon_2}{r_2}, 2\log{r_1}+2\ln{r_2}, -\frac{\epsilon_1}{r_1}-\frac{\epsilon_2}{r_2}\right) .
\]
\end{example}
\section{ Timelike minimal surface with an arbitrary number of weak complete ends} \label{section: arbitrary number of ends}
The construction of a complete genus zero maximal surface with an arbitrary number of ends has been studied by P. Kumar and S.R.R. Mohanty in \cite{kumar2023genus}. We aim to construct a timelike minimal surface with an arbitrary number of weak complete ends. We have divided this construction into two steps: the first step is to find suitable Weierstrass data for two maximal surfaces that solve the period condition on the punctured complex plane, which also satisfies the conjugate-symmetry property (see Remark \ref{conjugate-symmetric}). For this, we have to modify the Weierstrass data of a maximal surface as given in Section $3$ of \cite{kumar2023genus}.

% \subsection{Conjugate-symmetric Weierstrass data for maximal surfaces having an arbitrary number of ends on the real axis}

 Let $g_k$ be a conjugate-symmetric meromorphic function on $\mathbb{C}(\mathbf{i}) \cup \{\infty\} $ having poles only at the points $(p_k)_i$ of order $(x_k)_i$, where $k=1,\,2$ and $i = {1, \ldots, n}$. Without loss of generality, for each $k$ we can assume that $(p_k)_i$ is real for $i=1,\ldots,n-1$ and $(p_k)_n=\infty $. 
 Now, consider the two conjugate-symmetric holomorphic functions $f_1$ and $f_2$ as follows:
$$
f_1(\beta_1)=\sum_{i=1}^{n-1} \frac{(a_1)_i}{\left(\beta_1-(p_1)_i\right)^2}+\sum_{j=0}^{2 n-2} (b_1)_j \beta_1^j
$$\\
and
$$
f_2(\beta_2)=\sum_{i=1}^{n-1} \frac{(a_2)_i}{\left(\beta_2-(p_2)_i\right)^2}+\sum_{j=0}^{2 n-2} (b_2)_j \beta_2^j,
$$\\
 where ${(a_k)_i}, (b_k)_j$ are real numbers which are to be determined later.\\
 Next, we define $\tilde{g} := g_1 \mathbf{e} + g_2 \mathbf{e}^{\dagger}$ and $\tilde{f} := f_1 \mathbf{e} + f_2 \mathbf{e}^{\dagger}$, which satisfy the equation \eqref{5.4}, since $g_k, f_k$ (for $k = {1, 2}$) are conjugate-symmetric. At the same time, we take $(x_1)_i = (x_2)_i$ for all $i = {1, \ldots, n}$, which will ensure that $\tilde{g} = g_1 \mathbf{e} + g_2 \mathbf{e}^{\dagger}$ has a pole at $\hat{p}_i = (p_1)_i \mathbf{e} + (p_2)_i \mathbf{e}^{\dagger}$ of order $(x_1)_i$. Moreover, $\tilde{g}\bigr|_{\mathbb{D}} = \hat{g}$ has pole at $\hat{p}_i$ of order $(x_1)_i$.

Now only thing is to show that $(g_k, f_k)$ (for $k = {1, 2}$) are the Weierstrass data for some maximal surface $X_k$ on $\mathbb{C}(\mathbf{i}) \cup \{\infty\} \backslash \{(p_k)_1, \ldots, (p_k)_n\}$. For that, we need to show there exists a real vector $\left((a_k)_i, (b_k)_j\right)\in \mathbb{R}^{n-1}\times \mathbb{R}^{2n-1}$ such that the period condition, as in Theorem \ref{thm maximal}, holds on $\mathbb{C}(\mathbf{i}) \cup \{\infty\} \backslash \left\{(p_k)_1, \ldots, (p_k)_n\right\}$. The pair $\left(g_k, f_k\right)$ satisfies the period condition if the residues of $g_k^2 f_k$ and $g_k f_k$ vanish at the points $(p_k)_1,\ldots,(p_k)_{n-1}$. This gives us the following homogeneous system of linear equations in the variables $(a_k)_i$ and $(b_k)_j$:
\begin{equation}\label{residue}
\operatorname{Res}_{(p_k)_i}\left(g_k f_k\right)=0, \quad \operatorname{Res}_{(p_k)_i}\left(g_k^2 f_k\right)=0, \quad i=1, \ldots ,n-1; \quad k=1,2.
\end{equation}

  Since $g_k,\,f_k$ are conjugate-symmetric, the above system consists of at most $2n-2$ linear equations with real coefficients having $3n-2$ variables $(a_k)_i$ and $(b_k)_j$.  Therefore, the solution space for the system of equations \eqref{residue} has dimension at least $n$. This allows us to choose suitable non-zero vector $\left((a_k)_i , (b_k)_j\right)$ such that $\left(g_k, f_k\right)$ satisfies the period condition. Then from Proposition \ref{sum of two maximal surfaces} we get a timelike minimal surface with the Weierstrass data $(\tilde{g}\bigr|_{\mathbb{D}}, \tilde{f}\bigr|_{\mathbb{D}})$ in $\mathbb{D}\backslash (\{\hat{p}_1, \ldots,\hat{p}_{n-1}\}+S)$ having ends at $\hat{p}_i$.

It may happen that for some $i\,(i=1, \ldots ,n-1), (a_k)_i=0$, in which case $(p_k)_i$ may not be a pole of $f_k$. This implies $\tilde{f}\bigr|_{\mathbb{D}}$ may not have a pole at $\hat{p}_i$. Thus, while the pair $(\tilde{g}\bigr|_{\mathbb{D}},\tilde{f}\bigr|_{\mathbb{D}})$ gives a timelike minimal surface, it may not be weak complete around the neighbourhood of the end $\hat{p}_i$.
 Hence, we aim to modify $\tilde {f}$ such that $\tilde {p}_i$ becomes a weak complete end. Let us consider the following functions:
$$
(f_k)_i(\beta_k)=\sum_{l=1}^{2 n} \frac{(\alpha_k)_l}{\left(\beta_k-(p_k)_i\right)^{2+l}}
$$
 for $i=1, \ldots ,n-1$, where $\left((\alpha_k)_l\right) \in \mathbb{R}^{2 n}$ such that $\left(g_k, (f_k)_i\right)$ satisfies the period condition for maximal surfaces and  atleast one $(\alpha_k)_l \neq 0$. 
 Therefore, we have the following homogeneous system of linear equations in the variables $(\alpha_k)_l$:
\begin{equation*}
\operatorname{Res}_{(p_k)_j}\left(g_k (f_k)_i\right)=0, \quad \operatorname{Res}_{(p_k)_i}\left(g_k^2 (f_k)_i\right)=0, \quad j=1, \ldots ,n-1; \quad k=1,2.
\end{equation*}
This gives us $2 n-2$ independent linear homogeneous equations with $2 n$ variables. Hence, there exists  non-zero $(\alpha_k)_l$.
Now, if we define $F_k:=f_k+\sum_{i=1}^{n-1} (f_k)_i$, then $(g_k,\,F_k)$ will satisfy the period condition on $\mathbb{C}(\mathbf{i}) \cup\{\infty\} \backslash\left\{(p_k)_1, \ldots, (p_k)_n\right\}$ such that all $(p_k)_i$ 's are poles of $g_k$ and $F_k$. Thus, we have $\tilde{g}\bigr|_{\mathbb{D}}=\hat{g}$ and $\tilde{F}\bigr|_{\mathbb{D}}=\hat{F}$, both having poles at $\hat{p}_i$, ensuring all endpoints are weakly complete. Consequently, we arrive at the following theorem:

\begin{theorem} \label{main result}
    Given $k \in \mathbb{N}$ and $\hat{p}_1,\ldots,\hat{p}_k\in \mathbb{D}$, there exists a timelike minimal surface $\hat{X}:\mathbb{D}\backslash(\{\hat{p}_1,\ldots,\hat{p}_k\}+S)\rightarrow \mathbb{E}_1^3$ having $k$ weak complete ends at the points $\hat{p}_1,\ldots,\hat{p}_k$.
\end{theorem}

\begin{remark}
  Now, the constructed timelike minimal surface $\hat{X}$ with the Weierstrass data $(\hat{g},\hat{F} )$ is weak complete if the degenerate set of metric $\hat{h}_{\mathbb{E}^3}$ (see equation \eqref{weak complete}) is compact. Clearly, this is not the case unless we assume the domain of $\hat{X}$ to be bounded. 
\end{remark}

\section{Acknowledgement}
The authors are very grateful to Dr. Pradip Kumar for his valuable discussions. The first and third authors acknowledge the funding received from the University Grants Commission of India (UGC) under the UGC-JRF scheme. The second author is partially supported by an external grant from the Science and Engineering Research Board (SERB) (MATRICS, File No. MTR/2023/000990).

%%%%%%%%%%%%%%%%%THIRD CUT%%%%%%%%%%%%%%%%%%%%%%%

\bibliography{Priyank}
\bibliographystyle{ieeetr}

\end{document}